\documentclass[10pt]{amsart}
\usepackage{amsmath,amssymb,amsthm}
\usepackage{graphicx}
\newtheorem{theorem}{Theorem}    

\newtheorem{proposition}{Proposition} 
 
\newtheorem{question}{Question} 

\theoremstyle{definition}

\newtheorem{example}[theorem]{Example}
\newtheorem{remark}[theorem]{Remark}
\newtheorem*{remark*}{Remark}
\newcommand{\Z}{\mathbb{Z}}

\newcommand{\C}{\mathbb{C}}

\newcommand{\strutd}{ \,\raisebox{-3mm}{\includegraphics*[width=8mm]{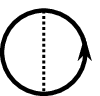}} }
\newcommand{\Lift}{\textrm{Lift}_{(p,q)}^{1-col}}

 \makeatletter
    
    \@addtoreset{equation}{section}
  \makeatother

\title[Mutation invariance of the zeroth coefficient polynomial]{Mutation invariance for the zeroth coefficients of colored HOMFLY polynomial}
\author[T.Ito]{Tetsuya Ito}
\address{Department of Mathematics, Kyoto University, Kyoto 606-8502, JAPAN}
\email{tetitoh@math.kyoto-u.ac.jp}

\begin{document}

\begin{abstract}
We show that the zeroth coefficient of the cables of the HOMFLY polynomial (colored HOMFLY polynomials) does not distinguish mutants. This makes a sharp contrast with the total HOMFLY polynomial whose 3-cables can distinguish mutants.
\end{abstract}

\maketitle

\section{Introduction}

Let $P_{K}(a,z)$ be the HOMFLY polynomial of an oriented knot or link $K$ in $S^{3}$ defined by the skein relation
\[ a^{-1}P_{\raisebox{-2mm}{\includegraphics*[width=3.5mm]{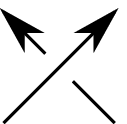}}}(a,z)+ a P_{\raisebox{-2mm}{\includegraphics*[width=3.5mm]{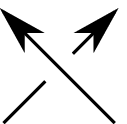}}}(a,z)= z P_{\raisebox{-2mm}{\includegraphics*[width=3.5mm]{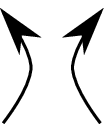}}}(a,z), \quad P_{\sf Unknot}(a,z) =1. \]
It is known that $P_{K}(a,z)$ is written in the form
\[ P_{K}(a,z)= (az)^{-r_K+1} \sum_{i\geq 0} \gamma_K^{i}(a)z^{2i}. \]
Here $r_K$ denotes the number of the components of $K$.
We call the polynomial $\gamma_K^{i}(a)$ the \emph{$i$-th coefficient (HOMFLY) polynomial} of $K$. 

One of an advantage of using the coefficient polynomials is that their skein formulae are simpler \cite{ka}. For example, the zeroth coefficient polynomial $\gamma^{0}_{K}(a)$ satisfies the skein formula
\[ a^{-2} \gamma^0_{\raisebox{-2mm}{\includegraphics*[width=3.5mm]{crossp.eps}}}(a)+ \gamma^{0}_{\raisebox{-2mm}{\includegraphics*[width=3.5mm]{crossn.eps}}}(a,z)= \begin{cases} \gamma^0_{\raisebox{-2mm}{\includegraphics*[width=3.5mm]{cross0.eps}}}(a) & \mbox{ if } \delta=0, \\
0 & \mbox{ otherwise}. 
\end{cases} \] 
Here $\delta=\frac{1}{2}(r_{ \raisebox{-2mm}{\includegraphics*[width=3.5mm]{crossp.eps}}} - r_{ \raisebox{-2mm}{\includegraphics*[width=3.5mm]{cross0.eps}}}+ 1) \in \{0,1\}$.
Consequently, the $i$-th coefficient polynomial $\gamma^{i}_K(a)$ for fixed $i$ is computed in polynomial time with respect to the number of crossings \cite{pr}. This makes a sharp contrast with the total HOMFLY polynomial whose computation is NP-hard \cite{j}.

Two knots $K$ and $K'$ are called \emph{mutant} if $K'$ is obtained from $K$ by \emph{mutation}. We take a ball $B$ whose boundary intersects with $K$ at four points to get a tangle $Q$. Then we cut $K$ along $Q$ and glue it back by rotating $Q$ (see Figure \ref{fig:mutation} (i)) to get another knot $K'$. If necessary, we change the orientation of strands in $Q$ so that the $K'$ admits a coherent orientation. 
\begin{figure}[htbp]
\begin{center}
\includegraphics*[width=90mm]{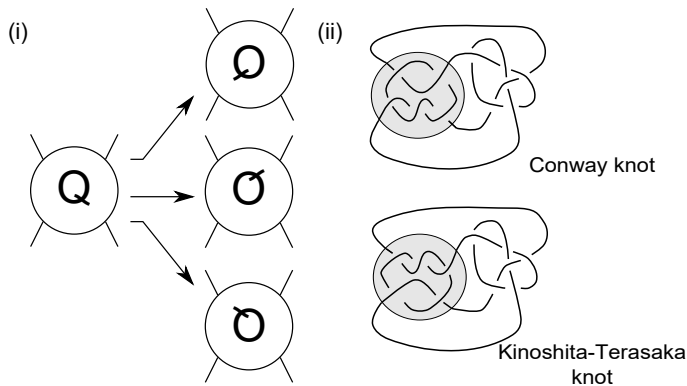}
   \caption{(i) Mutation (ii) Example of mutation: the Conway knot and the Kinoshita-Terasaka knot}
 \label{fig:mutation}
\end{center}
\end{figure}
It is often difficult to distinguish mutants since many familiar knot invariants, including the HOMFLY polynomial, take the same value for mutants.

For a coprime integers $p$ and $q$, we denote by $K_{p,q}$ the $(p,q)$-cable of the knot $K$. It may happen that the cables of a knot invariant $v$ can distinguish mutants, even if $v$ itself cannot. Namely, for mutants $K$ and $K'$, it may happen that $v(K)=v(K')$ but $v(K_{p,q}) \neq v(K'_{p,q})$. 

For example, the HOMFLY polynomial itself does not distinguish mutants but its 3-cable does; the HOMFLY polynomial of the 3-cable of the Conway knot and the Kinoshita-Terasaka knot (see Figure \ref{fig:mutation}(ii)), which are fundamental example of mutants, are different \cite{mc}.

Of course it may also happen that the cables of the invariant still fail to distinguish mutants -- for example, the cables of the Jones polynomial, the colored Jones polynomials, cannot distinguish mutants.

As shown in \cite{ta0}, the cables of the zeroth coefficient polynomial are not determined by the HOMFLY polynomial so it is interesting to ask the cables of zeroth coefficient polynomial can distinguish mutants or not. 

In this paper we show that the cables of the zeroth coefficient polynomial, the zeroth coefficient of the cables of the HOMFLY polynomial (so called the colored HOMFLY polynomial), do not distinguish mutants.

\begin{theorem}
\label{theorem:main}
If two knots $K$ and $K'$ are mutant, then $\gamma^0_{K_{p,q}}(a)=\gamma^{0}_{K'_{p,q}}(a)$. 
\end{theorem}

The mutation invariance for the cables of the zeroth coefficient polynomial was studied \cite{ta}, where it was shown that 3-cables of the zeroth coefficient polynomial do not distinguish mutants by skein theoretic argument. Our proof is based on a theory of Kontsevich invariants.

\section*{Acknowledgements}

This work was supported by JSPS KAKENHI Grant Number 15K17540, 16H02145. He thanks to H. Takioka for helpful conversations and anonymous referee for pointing out an error in the earlier version of the paper.

\section{The zeroth coefficient polynomial in the context of the Kontsevich invariant}

In this section we discuss the zeroth coefficient polynomial in a point of view of  quantum invariants. For basics of quantum invariants, we refer to \cite{oh}. We also refer to \cite{lm} for the relation between the HOMFLY polynomial and the Kontsevich invariant.

\subsection{Chord diagram and intersection graph}

A \emph{chord} over the circle $S^{1}$ is a pair of distinct points $\{x,y\} \subset S^{1}$. We call the points $x$ and $y$ the \emph{legs} of the chord. 
A \emph{chord diagram} is a correction of mutually distinct $n$ chords $\{\{x_i,y_i\}\}_{i=1,\ldots, n}$ over $S^{1}$. The \emph{degree} of the chord diagram $D$ is the number of chords, and we denote by $[D]$ the set of chords of $D$. As usual, we regard homeomorphic chord diagram as the same and express a chord diagram by drawing a diagram consisting of $S^{1}$ and chords connecting their legs.

The \emph{space of chord diagram} $\mathcal{C}=\mathcal{C}(S^{1})$ is a graded $\C$-vector space generated by chord diagrams, modulo 4T (four-term) relation in Figure \ref{fig:4T}. By taking the connected sum $\#$ as a multiplication, $\mathcal{C}$ is a graded commutative algebra. In the following, we will actually use the completion of $\mathcal{C}$ with respect to degrees, which we denote by the same symbol $\mathcal{C}$ by abuse of notation.

\begin{figure}[htbp]
\begin{center}
\includegraphics*[width=80mm]{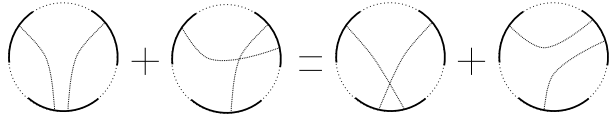}
   \caption{4T relation}
 \label{fig:4T}
\end{center}
\end{figure}

The \emph{intersection graph} of a chord diagram $D$ is a graph $\Gamma(D)$ whose vertices is $[D]$, the set of chords of $D$. Two vertices $v$ and $w$ are connected by an edge if and only if two chords $v$ and $w$ intersects. Here we say that chords $v=\{a,b\}$ and $w=\{c,d\}$ \emph{intersect} if the points $a$ and $b$ belong to the different components of $S^{1} \setminus \{c,d\}$.

For a non-negative integer $n$, we say that $D$ is \emph{$n$-colorable} if its intersection graph $\Gamma(D)$ is $n$-colorable. Namely, there is a map called a \emph{(vertex) coloring} $c: \{\mbox{vertices of }\Gamma(D) \} =[D] \rightarrow \{1,\ldots,n\}$ such that $c(v) \neq c(w)$ if $v$ and $w$ are connected by an edge.

\subsection{Information carried by the coefficient polynomials}

In the following, we will always regard a knot $K$ as a framed knot with 0-framing.

Let $V_N$ be the standard $N$-dimensional representation of the lie algebra $\mathfrak{sl}_N$. Let $Q^{(\mathfrak{sl}_N,V_N)}_K(q)$ be the quantum $(\mathfrak{sl}_N,V_N)$ invariant of a knot $K$, which is related to the HOMFLY polynomial by
\begin{equation}
\label{eqn:fund00}
Q^{(\mathfrak{sl}_N,V_N)}_K(q) = \frac{q^{\frac{N}{2}}-q^{-\frac{N}{2}}}{q^{\frac{1}{2}}-q^{-\frac{1}{2}}} P_{K}(q^{\frac{N}{2}},q^{\frac{1}{2}}-q^{-\frac{1}{2}}).
\end{equation}

The (framed) Kontsevich invariant $Z(K)$ is an invariant of framed knots that takes value in $\mathcal{C}$. The quantum $(\mathfrak{sl}_N,V_N)$ invariant is obtained from the Kontsevich invariant by a map $W_N=W_{(\mathfrak{sl}_N,V_N)}:\mathcal{C} \rightarrow \C[[h]]$, called the \emph{weight system} associated with $(\mathfrak{sl}_N,V_N)$, as
\begin{equation}
\label{eqn:fund01}
W_{N}(Z(K))= Q^{\mathfrak{sl}_N,V_N}_K(e^{h}).
\end{equation}
Combining (\ref{eqn:fund00}) and (\ref{eqn:fund01}), we get the following (cf. \cite[Theorem 2.3.1]{lm})
\begin{equation}
\label{eqn:fund1} W_{N}(Z(K)) = 
\frac{e^{\frac{Nh}{2}}-e^{-\frac{Nh}{2}}}{e^{\frac{h}{2}}-e^{-\frac{h}{2}}} P_{K}(e^{\frac{Nh}{2}},e^{\frac{h}{2}}-e^{-\frac{h}{2}}).
\end{equation}

We expand the right hand side of (\ref{eqn:fund1}) in a power series of $N$ and $h$ as
\begin{equation}
\label{eqn:fund02} W_N(Z(K))= \sum_{i,j} c_{i,j}(K) N^{i}h^{j}. 
\end{equation}
We also expand the $i$-th coefficient polynomial in a power series of $(Nh)$ by putting $a=e^{\frac{Nh}{2}}$, as
\[ \gamma_{K}^{i}(e^{\frac{Nh}{2}}) = \sum_{j=0}^{\infty} d_{i,j}(K)(Nh)^{j}.\]
Then we get an expansion 
\begin{align}
\nonumber
P_{K}(e^{\frac{Nh}{2}},e^{\frac{h}{2}}-e^{-\frac{h}{2}}) &= \sum_{i=0}^{\infty} \gamma^{i}_{K}(e^{\frac{Nh}{2}})(e^{\frac{h}{2}}-e^{-\frac{h}{2}})^{2i}\\
\nonumber & = \sum_{i=0}^{\infty}  \left(\sum_{j=0}^{\infty} d_{i,j}N^{j}h^{j}\right)  (h^{2i}+ (\deg h > 2i \mbox{ parts})) \\
\label{eqn:fund03}& = \sum_{j=0}^{\infty} N^{j}h^{j}( d_{0,j} + h(\mbox{power series on }h)) 
\end{align}

Since $\frac{e^{\frac{Nh}{2}}-e^{-\frac{Nh}{2}}}{e^{\frac{h}{2}}-e^{-\frac{h}{2}}}= N (1+ \mbox{power series on } N,h)$, by comparing (\ref{eqn:fund02}) and (\ref{eqn:fund03}) we get
\[
c_{i,j}(K)= \begin{cases}
 0 & i > j+1 \\
 d_{0,j}(K) & i=j+1
\end{cases}
\]
Thus, the zeroth coefficient polynomial $\gamma_{K}^{0}(a)$ is the diagonal part $c_{j+1,j}(K)$ of the HOMFLY polynomial, expanded as power series of $N$ and $h$ in (\ref{eqn:fund02}).

The same argument shows that the $i$-th coefficient polynomial $\gamma_{K}^{i}$ is determined by to the $i$-diagonal part $\{ c_{j+1,j}(K), c_{j+1,j+2}(K),\ldots, c_{j+1,j+2i}(K) \}_{j=0,\ldots}$ of the HOMFLY polynomial.

Let $\C[[N,h]]$ be the formal power series of $N$ and $h$ and 
$\Pi: \C[[N,h]] \rightarrow N\C[[hN]]$ be the projection.
We view $W_N$ as $W: \mathcal{C} \rightarrow \C[[N,h]]$ by viewing $N$ as variable and define
\[ W_{\gamma_0}=\Pi\circ W:\mathcal{C} \rightarrow N\C[[hN]]\]

Summarizing the argument so far, we conclude the following.
\begin{proposition}
\label{prop:fund}
For knots $K$ and $K'$, $\gamma_{K}^{0}(a)=\gamma_{K'}^{0}(a)$ if and only if
$W_{\gamma^0}(Z(K))=W_{\gamma^0}(Z(K'))$.
\end{proposition}

\begin{remark}
One may notice that the situation is similar to the Melvin-Morton-Rozansky  conjecture proven in \cite{bg}. We expand the colored Jones polynomials, quantum $\mathfrak{sl}_2$ invariant with respect to $N$-dimensional irreducible representation, as a power series of $h$ and $N$. Then the coefficients vanish under the diagonal part, and the diagonal part (so called the $1$-loop part) is equivalent to the Alexander-Conway polynomial. See \cite{bg} for details.
\end{remark} 

\subsection{Weight system associated with zeroth coefficient polynomial}
The weight system $W_N$ is described as state sum.
A \emph{state} of a chord diagram $D$ is a map $\sigma: [D] \rightarrow \{+1,-1\}$. For a state $\sigma$, we associate an oriented surface $S_\sigma$ as follows.
We start from the disk that bounds the outermost circle $S^{1}$ of the chord diagram. For a chord $c$ of $D$, if $\sigma(c)=+1$ we attach a 1-handle along $c$, and if $\sigma(c)=-1$ we do nothing (see Figure \ref{fig:state}).

\begin{figure}[htbp]
\begin{center}
\includegraphics*[width=115mm]{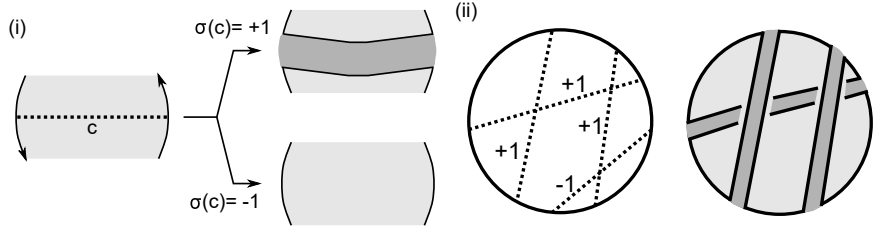}
   \caption{(i) Construction of surface $S_{\sigma}$. (ii) example.}
 \label{fig:state}
\end{center}
\end{figure}

Let $s(\sigma)$ be the number of boundary components of the resulting surface $S_{\sigma}$, and let $n(\sigma)$ be the number of chords such that $\sigma(c)=-1$. Then $W_{N}(D)$ is given by the following formula \cite{ba,cdm}.

\begin{equation}
\label{eqn:sl-weight} W_{N} (D) = \sum_{\sigma:[D] \rightarrow \{+1,-1\}}(-1)^{n(\sigma)}N^{s(\sigma)-n(\sigma)}h^{\deg D}
\end{equation}
Here $\sigma$ runs over all the states of $D$.

Let $\sigma_{+}$ be the state that assigns $+1$ for all the chords. We define the \emph{genus} $g(D)$ of $D$ as the genus of the surface $S_{\sigma_+}$. 
Then $s(\sigma_{+}) = \deg(D)+1-2g(D)$. 

For a state $\sigma$, changing the value of one chord from $+$ to $-$ can increase $s(\sigma)$ at most by one so we have $s(\sigma) \leq s(\sigma_+)+n(\sigma)$. Hence 
\[ s(\sigma)-n(\sigma) \leq s(\sigma_+) = \deg(D)-2g(D)+1. \]
Therefore $W_{N}(D)$ is of the form 
\[ W_N(D)= \sum_{0\leq j} x_{j} N^{\deg D+1-2g(D)-j}h^{\deg D} \qquad (x_j \in \Z).\]
This implies that, $W_{N}(D)$ contains a monomial of the form $N^{j}h^{j-1+2i}$ $(i,j\geq 0)$ only if $g(D) \leq i$. By Proposition \ref{prop:fund}, this means that the zeroth coefficient polynomial $\gamma^{0}_{K}(a)$ is determined by the genus $0$ chord diagram part of the Kontsevich invariant. 

Moreover, when $g(D)=0$ no two chords of $D$ intersect. Therefore in the case $g(D)=0$, for a state $\sigma$ we have $s(\sigma) \leq s(\sigma_+)$ and the equality occurs if and only if $\sigma=\sigma_{+}$. Thus if $g(D)=0$, $W_N(D)$ is of the form
\begin{equation}
\label{eqn:wsformula}
W_N(D)=N^{\deg D + 1}h^{\deg D} + \sum_{1\leq j} x_{j} N^{\deg D+1-j}h^{\deg D}. 
\end{equation}
These considerations show that the following simple description of $W_{\gamma^0}$. 
\begin{proposition}
\label{prop:weightformula}
For a chord diagram $D$,
\[ W_{\gamma^0}(D) =\begin{cases} N^{\deg D +1} h^{\deg D} & (g(D)=0)\\
 0 & (\mbox{otherwise})
\end{cases}
\]
\end{proposition}

\section{Kontsevich invariant of mutants and cables}

In this section we review a formula of the Kontsevich invariant of mutants and cables.

\subsection{Mutation of chord diagram and the Kontsevich invariant of mutants}

We say that a mutation is \emph{of type A} if it preserves the strands in the tangle setwise (see Figure \ref{fig:mutation2} (i)). Other two types of mutations are achieved by suitable compositions of mutation of type A (see Figure \ref{fig:mutation2} (ii)). So in the following we will concentrate our attention to the mutation of type A.

\begin{figure}[htbp]
\begin{center}
\includegraphics*[width=100mm]{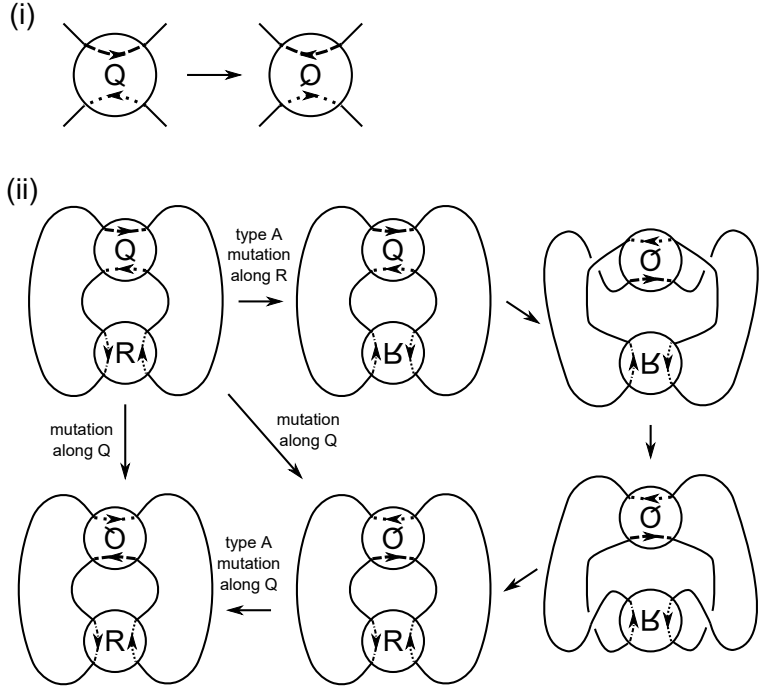}
   \caption{(i) mutation of type A. (ii) other types of mutations are achieved by mutation of type A}
 \label{fig:mutation2}
\end{center}
\end{figure}

A \emph{share} $S(I,J)$ in a chord diagram $D$ is a union of disjoint intervals $I$ and $J$ in $S^{1}$ and chords of $D$ both of whose legs lie on $I \cup J$. 
A \emph{mutation} of the chord diagram $D$ along a share $S$ is a chord diagram obtained by rotating or flyping (or both) a share $S$ (see Figure \ref{fig:share_mutation}). 
We say that a mutation of chord diagram is \emph{of type A} if it flips the orientation of $I \cup J$ but it does not swap the position of $I$ and $J$. This corresponds to the type A mutation of a knot.

\begin{figure}[htbp]
\begin{center}
\includegraphics*[width=80mm]{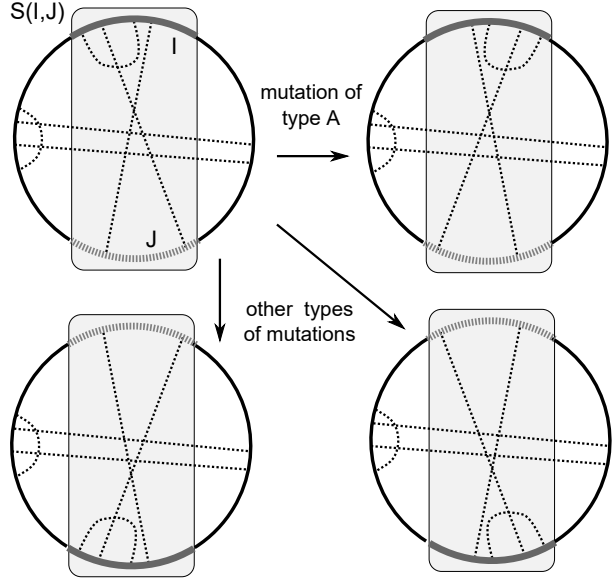}
   \caption{Share and mutation of chord diagrams}
 \label{fig:share_mutation}
\end{center}
\end{figure}

For mutants $K$ and $K'$, the Kontsevich invariant $Z(K')$ is obtained from $Z(K)$ by applying the mutation of each chord diagrams in $Z(K)$ along an appropriate share that corresponds to the mutation tangle $Q$ (see \cite{cdm} for precise statement). This gives the following criterion for mutation invariance of knot invariants.

\begin{proposition}
Let $V$ be a knot invariant that is governed by the Kontsevich invariant. That is, there is a map $W_V:\mathcal{C} \rightarrow \C$ such that $V(K)=W_V(Z(K))$.
Assume that if $D$ and $D'$ are related by mutation of chord diagrams, $W_{V}(D)=W_{V}(D')$. Then $V$ does not distinguish mutants.
\end{proposition}

Actually, more is true; two chord diagram have the same intersection graph if and only if they are related by mutations. Moreover, a knot invariant $V$ does not distinguish mutant if and only if $W_V$ only depends on the intersection graph \cite{cl}.

\subsection{Kontesvich invariant of cables}

The effect of the cabling  for the Kontsevich invariant is described as follows \cite{blt,lm2}.

For a pair of integers $(p,q)$, let $l=l_{p,q}$ be a simple closed curve on a torus $T=S^{1} \times S^{1}$ that represents $p[S^{1}\times \{\ast \}] + q[\{\ast\} \times S^{1}] \in H_{1}(T;\Z)$. Let $\pi=S^{1}\times S^{1} \rightarrow S^{1}$ be the projection map to the first factor. The restriction of $\pi$ to $l$ gives a $p$-fold covering map $\pi=\pi_{p,q}:l=\widetilde{S^{1}} \rightarrow S^{1}$. 
In the following, to distinguish the base circle $S^{1}$ and its $p$-fold covering, we denote the covering circle by $\widetilde{S^{1}}$.

For a chord diagram $D$ on $S^{1}$, we say that a chord diagram $\widetilde{D}$ on $\widetilde{S}^{1}$ is a \emph{lift} of $D$ if $\pi(\widetilde{D})=D$.
That is, for $D=\{\{x_1,y_1\},\ldots,\{x_n,y_n\}\}$, $\widetilde{D}$ is a lift of $D$ if and only if $\widetilde{D}=\{\{\widetilde{x_1},\widetilde{y_1}\},\ldots,\{\widetilde{x_n},\widetilde{y_n}\}\}$ where $\widetilde{x_i} \in \pi^{-1}(x_i)$ and $\widetilde{y_i} \in \pi^{-1}(y_i)$ for each $i$.

We take a base point $\ast \in S^{1}$ so that $\ast$ is different from all the legs of chords of $D$. Then $\pi^{-1}(S^{1} \setminus \ast)$ is a union of disjoint intervals $L_1,\ldots,L_p$. For $x \in S^{1} \setminus \ast$, we denote the lift of the point $x$ which lies on $L_i$ by $x_i$. 

We call a map $s:\{ \mbox{Set of legs of chords of } D\} \rightarrow \{1,\ldots,p\}$ a \emph{leg $p$-coloring}. There is a one-to-one correspondence between the set of all the leg $p$-colorings and the set of all the lifts $\widetilde{D}$ of $D$: for a leg $p$-coloring $s$, we assign the lift $D_{s}$ so that the chord $\{x,y\}$ is lifted to the chord $\{x_{s(x)},y_{s(y)}\}$ (see Figure \ref{fig:lift}).

\begin{figure}[htbp]
\begin{center}
\includegraphics*[width=100mm]{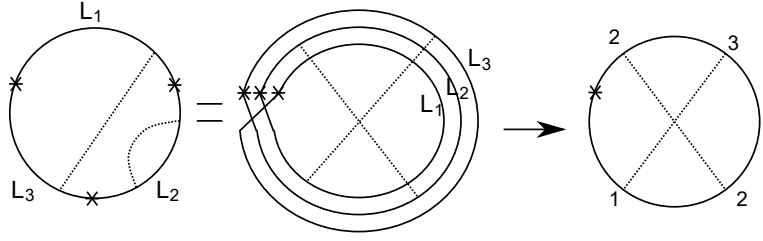}
   \caption{Lift of chord diagrams and leg $p$-colorings}
 \label{fig:lift}
\end{center}
\end{figure}

We define the map $\psi^{(p|q)}: (\mathcal{C}(S^{1}) =) \mathcal{C} \rightarrow \mathcal{C} (=\mathcal{C}(\widetilde{S^{1}}))$ by
\[ \psi^{(p|q)}(D)= \sum_{s} \widetilde{D}_s \]
where $s$ runs all the leg $p$-coloring (that is, $\psi^{(p|q)}(D)$ is the sum of all the lifts of $D$).
Using the map $\psi^{(p|q)}$, the cabling formula of Kontsevich invariant is written as follows \cite[Theorem 1, Remark 3.4]{blt}:
\begin{equation}
\label{eqn:cabling}
 Z(K_{(p,q)}) = \left[ \psi^{(p|q)}\left( Z(K) \# \exp(\frac{q}{2p}
\strutd )\right) \right] \# \exp(-\frac{q}{2} 
\strutd)
\end{equation}

Here $\exp(D)=1+D+\frac{1}{2}D\# D + \cdots = \sum_{n=0}^{\infty}\frac{1}{n!} \underbrace{D\# \cdots \# D}_{n}$.

If $D$ is $p$-colorable, then there exists a 1-colorable lift $\widetilde{D}$ of $D$: For a $p$-coloring of $D$, we take a leg $p$-coloring $s$ of $D$ given by $s(x_i)=s(y_i)=c$ if the chord $\{x_i,y_i\}$ is colored by $c$. Then its corresponding lift $\widetilde{D}_s$ is $1$-colorable.
It is interesting to ask whether the converse is true or not:

\begin{question}
\label{ques:chord}
Is a lift $\widetilde{D}$ of $D$ admits a $1$-coloring, if and only if $D$ is $p$-colorable ?
\end{question}

Since $g(D)=0$ if and only if $D$ is 1-colorable, it is the $1$-colorable chord diagrams that contributes to the zeroth coefficient polynomial.

An affirmative answer to Question \ref{ques:chord} identifies the part of the Kontsevich invariant reflected to the cable of the zeroth coefficient polynomial.
The zeroth coefficient polynomial of the $(p,q)$-cable is determined by the $i$-colorable ($i\leq p$) chord diagram part of the Kontsevich invariant. 

\section{Proof of invariance}

For coprime $p,q$, let $\pi_{p,q}:\widetilde{S^{1}}\rightarrow S^{1}$ be the $p$-fold covering taken in the previous section.
We denote by $\Lift(D)$ be the set of $1$-colorable lift of the chord diagram $D$ to $\widetilde{S^{1}}$. 

Let $D^{\tau}$ be the chord diagram that is obtained from $D$ by mutation of type A along a share $S=S(I,J)$. The mutation invariance of cables of zeroth coefficient polynomial follows from the following assertion.

\begin{proposition}
\label{proposition:key}
There is a one-to-one correspondence between $\Lift(D)$ and $\Lift(D^{\tau})$. 
\end{proposition}
\begin{proof}
We show an existence of one-to-one correspondence by constructing injections $\Psi:\Lift(D) \rightarrow \Lift(D^{\tau})$ and $\Phi:\Lift(D^{\tau}) \rightarrow \Lift(D)$.
 To make notation simpler in the following we consider the case $q=1$. A general case is similar.

Let $X,Y$ be the connected components of $S^{1} \setminus (I \cup J)$. We take a basepoint $\ast \in S^{1}$ as one of the boundary of $I$ so that four intervals $I,X,J,Y$ appear in this order along $S^{1}$. 
We denote the $i$-th lift of the interval $L \in \{I,J,X,Y\}$ by $L_{i}$. 
We give the total ordering $<$ on the set $\{I_i,J_i,X_i,Y_i\}$ according to the orientation of $\widetilde{S^{1}}$. Namely, $<$ is defined by
\[ I_1 < X_1 < J_1 < Y_1 < I_2 < \cdots < I_p < X_p < J_p < Y_p.\]

Take a $1$-colorable lift $\widetilde{D} \in \Lift(D)$. For $A,B \in \{I_1,\ldots, I_p,J_1,\ldots,J_p\}$ we define $A \sim B$ if there is a chord connecting $A$ and $B$. By abuse of notation, we use the same symbol $\sim$ to express the equivalence relation generated by $\sim$. We decompose the set of lifted intervals into the disjoint union of equivalence classes
\[ \{I_1,\ldots, I_p,J_1,\ldots,J_p\} = R_1\sqcup R_2 \sqcup \cdots \sqcup R_k.\]
By the same manner, we decompose 
\[ \{X_1,\ldots, X_p,Y_1,\ldots,Y_p\} = S_1\sqcup S_2 \sqcup \cdots \sqcup S_l\] 
W choose the numbering of equivalence class so that
\[ \min_< \{L \in R_1\} < \min_< \{L \in R_2\} < \cdots \]
holds.

We view each equivalence class $R_i$ (or $S_i$) as a sub-chord diagram of $\widetilde{D}$, and we view $\widetilde{D}$ as a union of sub-chord diagrams $R_1,\ldots,R_k,S_1,\ldots,S_l$ (See Figure \ref{fig:flypeeg0}). 

\begin{figure}[htbp]
\begin{center}
\includegraphics*[width=100mm]{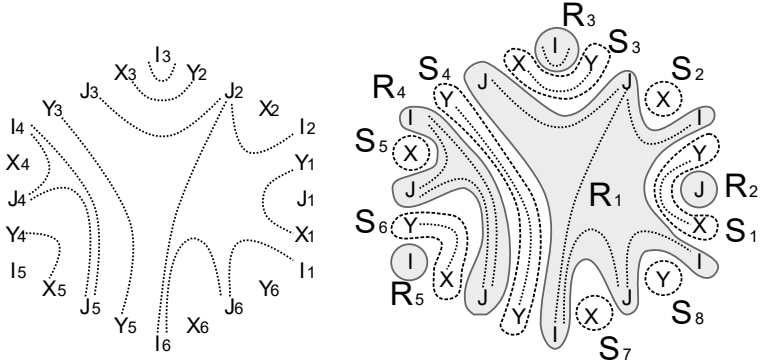}
   \caption{Lift of chord diagram $\widetilde{D}$ and factorization into sub-chord diagrams.}
 \label{fig:flypeeg0}
\end{center}
\end{figure}

We align an element of $R_i$ according to the order $<$. Let $W(R_i)$ be the word over $\{I,J\}$ obtained by reading the letter $\{I,J\}$ of $R_i$. For example, if $R_i=\{I_1,I_2,J_4,J_6,I_8,J_8\}$ then $W(R_i)=IIJJIJ$.

Then we construct a new chord diagram $f_1(\widetilde{D})$, which we call  \emph{flipping $R_1$ of $\widetilde{D}$} as follows (see Example \ref{exam:flip} below for detailed illustration).

Let $W=W(R_1)$. We construct the following four ordered sets of words $\mathcal{W}_{I\to I}$, $\mathcal{W}_{I\to J}$, $\mathcal{W}_{J\to I}$, and $\mathcal{W}_{J\to J}$ as follows.
For $s=1,\ldots,\#R_1(=\mbox{length of } W$, we look at the $s$-th and $(s+1)$-st letter of $W$. Here we regard $W$ as a cyclic word so when $s=\#R_1$ we look at the last and the first letter of $W$.
There are four possibilities; $(I,I)$, $(I,J)$, $(J,I)$ and $(J,J)$.
Assume that the $s$-th letter is $I=I_a$ and $(s+1)$st letter is $J=J_b$.
Then we add the word $X_{a}J_{a}Y_{a}I_{a+1}X_{a+1}\cdots I_{b}X_{b} = XJYI\cdots IX$ to the set $\mathcal{W}_{I\to J}$. The other cases are similar.

Let $\overleftarrow{W}$ be the reverse of $W$, the word obtained from $W$ by reading it in the backward direction. For example, if $W=IIJJIJ$, $\overleftarrow{W}=JIJJII$. 
Then we construct $\overleftarrow{R_1} \subset \{I_*,J_*,X_*,Y_*\: | \: *=1,\ldots,p\}$ so that the corresponding word $W(\overleftarrow{R_1})$ is $\overleftarrow{W}$.

As the first step we put $\overleftarrow{R_1}=\emptyset$. 
For $i=1,\ldots,\#R_1=\mbox{length of } \overleftarrow{W}$, we add $I_{*}$ or $J_{*}$ to $\overleftarrow{R_1}$ according to the following inductive way. 

For $i=1$ we add $I_1$ or $J_1$ to $\overleftarrow{R_1}$ according to the first letter of $\overleftarrow{W}$; If the first letter of $\overleftarrow{W}$ is $I$ (resp. $J$) then we add $I_1$ (resp. $J_1$). 

Assume that the $i$-th letter is $I$ and $(i+1)$st letter is $J$, and that we added $I_{a}$ to the set $\overleftarrow{R_1}$ in the previous step. 
Then we add $J_b$ to $\overleftarrow{W}$ so that the `gap', the word between $I_a$ and $J_b$
\[ X_{a}J_{a}Y_a\cdots I_b X_b = XJY \cdots IX  \]
coincides with the first element of $\mathcal{W}_{I\to J}$. Then we remove this word $XJY \cdots YJ$ from the set $\mathcal{W}(I\to J)$. The other cases are treated similarly.

Using $\overleftarrow{W}$, $\mathcal{W}_{I\to I}$, $\mathcal{W}_{I\to J}$, $\mathcal{W}_{J\to I}$, and $\mathcal{W}_{J\to J}$ we construct a new chord diagram $f_1(\widetilde{D})$ by cut-and-paste.

We view the chord diagram $\widetilde{D}$ as a union of a sub-chord diagram $R_1$ and its complements $C_1,\ldots,C_n$. Then there is a one-to-one correspondence between 
\[ \{C_1,\ldots,C_n\} \mbox{ and } \mathcal{W}_{I\to I} \cup \mathcal{W}_{I\to J} \cup \mathcal{W}_{J\to I} \cup \mathcal{W}_{J\to J}.\]

We flip the sub-chord diagram $R_1$ to reverse the orientation of the boundary intervals in $R_1$. Then we glue the flipped sub-chord diagram back to $\widetilde{S^{1}}$ so that its $i$-th interval is glued to the $(\# R_1 -i)$-th interval of $\overleftarrow{R_1}$. 

Then we glue back the remaining sub-chord diagrams $C_1,\ldots,C_n$ of $R_1$ to the complementary regions, using the one-to-one correspondence of two sets
\[ \{C_1,\ldots,C_n\} \mbox{ and } \mathcal{W}_{I\to I} \cup \mathcal{W}_{I\to J} \cup \mathcal{W}_{J\to I} \cup \mathcal{W}_{J\to J}.\] 
We denote the resulting chord diagram by $f_1(\widetilde{D})$.

The chord diagram $f_{i}(\widetilde{D})$, \emph{flipping $R_i$ of $\widetilde{R}$} are defined similarly. By construction, this operation is injective in the sense
\begin{equation}
\label{eqn:inj}
f_i(\widetilde{D})=f_i(\widetilde{D'}) \mbox{ if and only if } \widetilde{D}=\widetilde{D'}.
\end{equation}

Let $\Psi(\widetilde{D})=f_{k}\circ \cdots \circ f_2 \circ f_1(\widetilde{D})$. 
Since $\Psi$ flips all the sub-chord diagrams $R_i$ of $\widetilde{D}$ preserving the orientation of other chord diagrams $S_i$, $\Psi(\widetilde{D}) \in \Lift(D^{\tau})$. By the property (\ref{eqn:inj}), $\Psi: \Lift(D) \rightarrow \Lift(D^{\tau})$ is injective. By interchange the role of $D$ and $D^\tau$, we also have an injective map $\Phi: \Lift(D^{\tau}) \rightarrow \Lift(D)$. This shows that two sets $\Lift(D)$ and $\Lift(D^{\tau})$ have the same cardinal.
\end{proof}

\begin{example}[Flipping the sub-chord diagram $R_1$ of $\widetilde{R}$]
\label{exam:flip}
Here we give an example of flipping the sub-chord diagram $R_1$ of $\widetilde{R}$. See Figure \ref{fig:flype-exam1}. Assume that $R_1=\{I_1,I_3,J_3,J_6,I_8,J_8\}$. 

By definition, $W=W(R_1)=IIJJIJ$. The 1st letter of $W$ is $I$, the 2nd letter of $W$ is $I$. The word obtained by reading intervals between $I_1$ and $I_3$ is $X_1J_1Y_1I_2X_2J_2Y_2$ so we add the word $XJYIXJY$ to the set $\mathcal{W}_{I\to I}$. Similarly, the 3rd letter is $J$ so we add the word $X$ to $\mathcal{W}_{I\to J}$. Iterating this procedure, we get 
\[ \begin{cases} \mathcal{W}_{I\to I}=(XJYIXJY)\\
\mathcal{W}_{I\to J}=(X,XJYIX)\\ 
\mathcal{W}_{J\to I}=(YIXJY,YIXJY)\\ 
\mathcal{W}_{J\to J}=(YIX).
\end{cases}
\]

Then the construction of $\overleftarrow{R_1}$ goes as follows. By definition, $\overleftarrow{W}=JIJJII$. We start from the emptyset. The first letter of $\overleftarrow{W}$ is $J$ so we add $J_1$ to $\overleftarrow{R_1}$ (so now $\overleftarrow{R_1}=\{J_1\}$). The second letter of $\overleftarrow{W}$ is $I$ so we look at the set $\mathcal{W}_{J\to I}=(YIXJY,YIXJY)$. Its first element is $YIXJY$. So we add $I_3$ to $\overleftarrow{R_1}$ since the `gap' between $J_1$ and $I_3$ is $Y_1I_2X_2J_2Y_2=YIXJY$ (so now $\overleftarrow{R_1}=\{J_1,I_3\}$), and we remove the word $YIXJY$ from $\mathcal{W}_{J\to I}$ (so now $\mathcal{W}_{J\to I}=(YIXJY)$).
Similarly, the 3rd letter is $J$ (and the 2nd letter is $I$) so we look at the set $\mathcal{W}_{I\to J} =(X,XJYIX)$. Its first element is $X$. We add $J_3$ to $\overleftarrow{R_1}$ (so now $\overleftarrow{R_1}=\{J_1,I_3,J_3\}$) and remove the word $X$ from $\mathcal{W}_{I\to J}$ (so now $\mathcal{W}_{I\to J} =(XJYIX)$).
Iterating this procedure we get 
\[ \overleftarrow{R_1}=\{J_1,I_3,J_3,J_4,I_6,I_8\}.\]

Then we construct the chord diagram $f_1(\widetilde{D})$. We flip the sub-chord diagram $R_1$ and glue it back. We glue the boundary intervals $(I_1,I_3,J_3,J_6,I_8,J_8)$ of $R_1$ to $(I_8,I_6,J_4,J_3,I_3,J_1)$ respectively with reversing the orientations.

We glue back the complementary sub-chord diagrams $C_1,\ldots,C_6$. The boundary intervals of the sub-chord diagram $C_1$ is $X_1J_1Y_1I_2X_2J_2Y_2$, which corresponds to the first element $XJYIXJY$ in $\mathcal{W}_{I\to I}$. In the construction of $\overleftarrow{R_1}$ we used this word $XJYIXJY$ as a gap between $I_6$ and $I_8$. So we glue $C_1$ so that the intervals $X_1,J_1,\ldots,J_2,Y_2$ are identified with $X_6,J_6,\ldots,J_7,Y_7$ (preserving the orientation). Other sub-chord diagram are glued similarly. For example, $C_6$ corresponds to the second element $YIXJY$ of $\mathcal{W}_{J\to I}$, which we used as a gap between $J_4$ and $I_6$. Hence $C_6$ is glued back so that intervals $Y_7,I_8,X_8,J_8,Y_8$ are identified with the intervals $Y_4,I_5,X_5,J_5,Y_5$, respectively. \\

\begin{figure}[htbp]
\begin{center}
\includegraphics*[width=100mm]{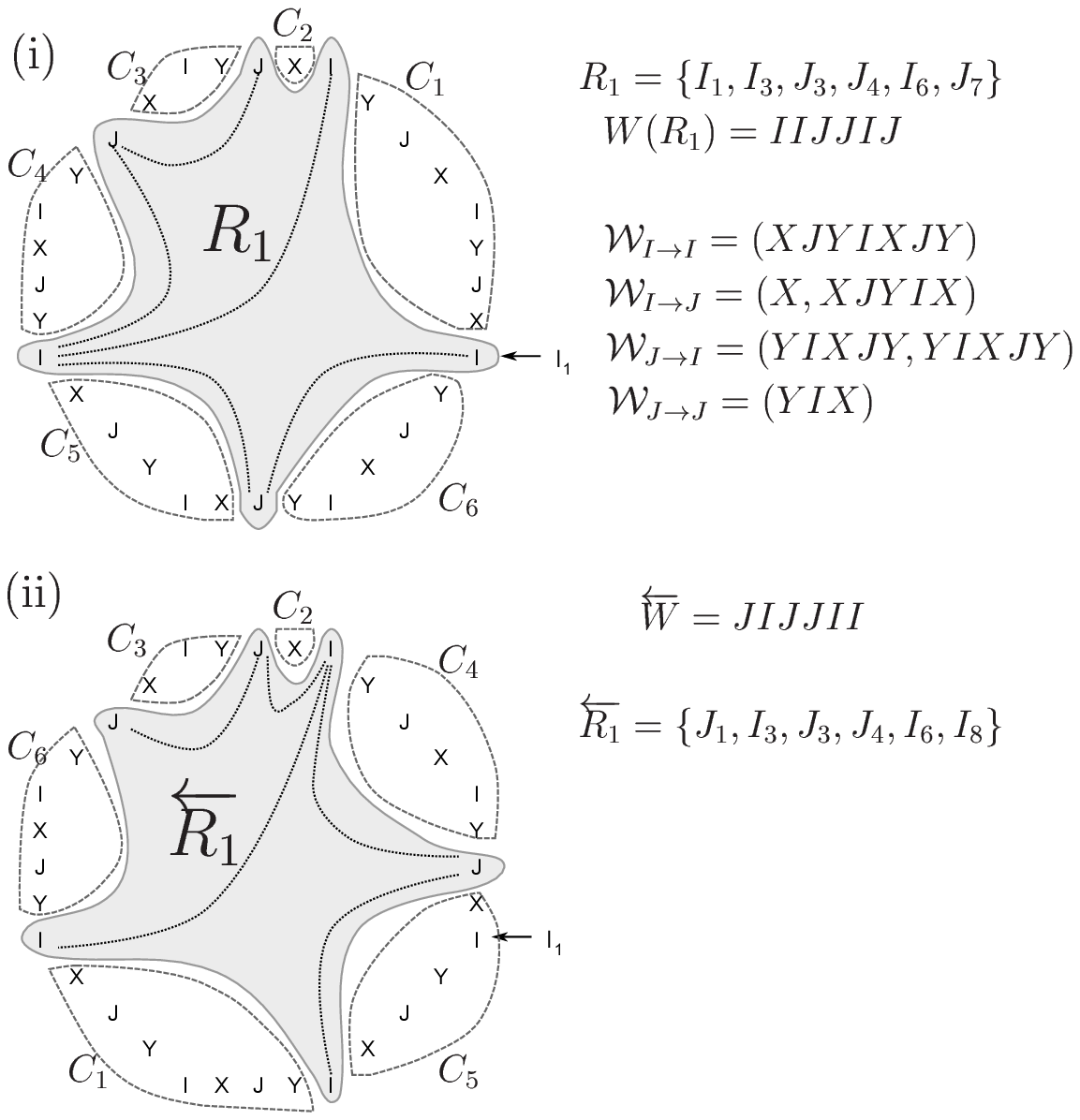}
   \caption{Flipping sub-chord diagram $R_1$: (i) describes $\widetilde{D}$ and (ii) describes $f_1(\widetilde{D})$ (we omit to describe chords in $C_1,\ldots,C_6$).}
 \label{fig:flype-exam1}
\end{center}
\end{figure}
\end{example}

\begin{proof}[Proof of Theorem \ref{theorem:main}]
Let $K$ and $K'$ be mutants. With no loss of generality we may assume that $K$ and $K'$ are related by mutation of type A. 

For a chord diagram $D$ let $D^{\tau}$ be the corresponding mutation. 
By Proposition \ref{prop:weightformula}, for a chord diagram $D$ we have
\[ W_{\gamma^0}(\psi^{(p|q)}(D))= \# \Lift(D) N^{\deg D +1}h^{\deg D} \]
By Proposition \ref{proposition:key}, $\# \Lift(D) =\# \Lift(D^\tau)$ hence
\[  W_{\gamma^0}(\psi^{(p|q)}(D)) =  W_{\gamma^0}(\psi^{(p|q)}(D^{\tau}))\]
By the cabling formula (\ref{eqn:cabling}) this implies $W_{\gamma^0}(Z(K_{p,q}))=W_{\gamma^0}(Z(K'_{p,q}))$ hence Proposition \ref{prop:fund} shows $\gamma_{K_{p,q}}^0(a)=\gamma_{K'_{p,q}}^0(a)$ as desired.
\end{proof}


\begin{thebibliography}{1}
\bibitem[Ba]{ba}
D. Bar-Natan,
{\em On the Vassiliev knot invariants,} 
Topology \textbf{34} (1995), 423--472. 

\bibitem[BG]{bg}
D.\ Bar-Natan and S. Garoufalidis, {\em On the Melvin-Morton-Rozansky conjecture}, Invent. Math, \textbf{125}, (1996), 103--133.

\bibitem[BLT]{blt} D.\ Bar-Natan, T.\ Le, and D.\ Thurston,
{\em Two applications of elementary knot theory to Lie algebras and Vassiliev invariants}, Geom. Topol. \textbf{7} (2003), 1--31. 

\bibitem[CDM]{cdm}
S. Chmutov, S. Duzhin and J. Mostovoy,
{\em Introduction to Vassiliev knot invariants,}
Cambridge University Press, Cambridge, 2012. xvi+504 pp.
\bibitem[CL]{cl}
S. Chmutov, and S. Lando, 
{\em Mutant knots and intersection graphs,}
Algebr. Geom. Topol. \textbf{7} (2007), 1579--1598. 
\bibitem[J]{j} F. Jarger,
{\em Tutte polynomials and link polynomials,}
Proc. Amer. Math. Soc. \textbf{103} (1988), 647--654.
\bibitem[Ka]{ka} A. Kawauchi,
{\em On coefficient polynomials of the skein polynomial of an oriented link}, 
Kobe J. Math. \textbf{11} (1994), 49--68. 

\bibitem[LM1]{lm} T. Le and J. Murakami, 
{\em Kontsevich's integral for the Homfly polynomial and relations between values of multiple zeta functions,}
Topology Appl. \textbf{62} (1995), 193--206. 

\bibitem[LM2]{lm2} T. Le and J. Murakami, 
{\em Parallel version of the universal Vassiliev-Kontsevich invariant},
J. Pure Appl. Algebra 121 (1997), 271--291. 

\bibitem[LL]{ll} R. Lickorish and S. Lipson,
{\em Polynomials of 2-cable-like links,} 
Proc. Amer. Math. Soc. \textbf{100} (1987), 355--361. 

\bibitem[MC]{mc}
H. Morton and P. Cromwell,
Distinguishing mutants by knot polynomials.  
J. Knot Theory Ramifications \textbf{5} (1996), 225--238. 
\bibitem[Oh]{oh} T. Ohtsuki,
{\em Quantum invariants,}
Series on Knots and Everything, 29. World Scientific Publishing Co., Inc., River Edge, NJ, 2002.
\bibitem[Pr]{pr} J. Przytycki,
{\em The first coefficient of Homflypt and Kauffman polynomials: Vertigan proof of polynomial complexity using dynamic programming,}
Contemp. Math. \textbf{689} (2017) 1--6.

\bibitem[Ta1]{ta0} H. Takioka, 
{\em The zeroth coefficient HOMFLYPT polynomial of a 2-cable knot} 
J. Knot Theory Ramifications 22 (2013), 1350001, 25 pp. 

\bibitem[Ta2]{ta} H. Takioka, 
{\em The cable $\Gamma$-polynomials of mutant knots,} 
Topology Appl. \textbf{196} (2015), part B, 911--920. 


\end{thebibliography}
\end{document}